\documentclass[11pt]{article}

\usepackage{amsmath, amsfonts,amssymb, amsthm}
\usepackage{graphicx}
\usepackage{color}
\usepackage{epsf}
\usepackage{tikz}
\usepackage{subfig}
\usepackage{hyperref}
\hypersetup{
    colorlinks = true,
    linkcolor = {black},
    citecolor = {black}
}

\usepackage[left=1.3in, right=1.3in, top=1.2in, bottom=1.25in]{geometry}

\usepackage{enumitem}
\setlist{itemsep=0pt, topsep=0pt}

\newcommand{\floor}[1]{\lfloor#1\rfloor}
\newcommand{\ceiling}[1]{\lceil#1\rceil}

\newcommand{\tbf}[1]{\textbf{#1}}

\newtheorem{theorem}{Theorem}[section]
\newtheorem{corollary}[theorem]{Corollary}
\newtheorem{lemma}[theorem]{Lemma}

\newtheorem{claim}[theorem]{Claim}
\newtheorem{proposition}[theorem]{Proposition}
\newtheorem{observation}[theorem]{Observation}

\newtheorem{example}[theorem]{Example}
\newtheorem{problem}[theorem]{Problem}

\newtheorem{conjecture}[theorem]{Conjecture}

\newenvironment{proofclaim}[1][Proof of claim]{\begin{proof}[#1]}{\end{proof}}

\newcommand{\gnp}{G(n,p)}

\newcommand{\hnpr}{H^r(n,p)}

\newcommand{\mc}{\mathrm{mc}}
\newcommand{\halpha}{\hat{\alpha}}

\newcommand{\of}[1]{\left( #1 \right)}

\title{Large monochromatic components in expansive hypergraphs}
\author{Deepak Bal\thanks{Department of Mathematics, Montclair State University, Montclair, NJ {\tt deepak.bal@montclair.edu}.}  \and Louis DeBiasio\thanks{Department of Mathematics, Miami University, Oxford, OH. \texttt{debiasld@miamioh.edu}. Research supported in part by NSF grant DMS-1954170.}}
\date{\today}

\begin{document}

\maketitle

\begin{abstract}
A result of Gy\'arf\'as \cite{Gy} exactly determines the size of a largest monochromatic component in an arbitrary $r$-coloring of the complete $k$-uniform hypergraph $K_n^k$ when $k\geq 2$ and $r-1\leq k\leq r$.  We prove a result which says that if one replaces $K_n^k$ in Gy\'arf\'as' theorem by any ``expansive'' $k$-uniform hypergraph on $n$ vertices (that is, a $k$-uniform hypergraph $H$ on $n$ vertices in which in which $e(V_1, \dots, V_k)>0$ for all disjoint sets $V_1, \dots, V_k\subseteq V(H)$ with $|V_i|>\alpha$ for all $i\in [k]$), then one gets a largest monochromatic component of essentially the same size (within a small error term depending on $r$ and $\alpha$).  As corollaries we recover a number of known results about large monochromatic components in random hypergraphs and random Steiner triple systems, often with drastically improved bounds on the error terms.

Gy\'arf\'as' result is equivalent to the dual problem of determining the smallest maximum degree of an arbitrary $r$-partite $r$-uniform hypergraph with $n$ edges in which every set of $k$ edges has a common intersection.  In this language, our result says that if one replaces the condition that every set of $k$ edges has a common intersection with the condition that for every collection of $k$ disjoint sets $E_1, \dots, E_k\subseteq E(H)$ with $|E_i|>\alpha$ for all $i\in [k]$ there exists $e_i\in E_i$ for all $i\in [k]$ such that $e_1\cap \dots \cap e_k\neq \emptyset$, then the maximum degree of $H$ is essentially the same  (within a small error term depending on $r$ and $\alpha$).  We prove our results in this dual setting.  
\end{abstract}

\section{Introduction}

We say that a hypergraph $H$ is connected if the 2-shadow of $H$ is connected (the 2-shadow of $H$ is the graph on vertex set $V(H)$ and edge set $\{e\in \binom{V(H)}{2}: \exists f\in E(H), e\subseteq f\}$).  A component in a hypergraph is a maximal connected subgraph.  Given a hypergraph $H$ and a positive integer $r$, let $\mc_r(H)$ be the largest integer $t$ such that every $r$-coloring of the edges of $H$ contains a monochromatic component of order at least $t$.  Let $K_n^k$ denote the complete $k$-uniform hypergraph on $n$ vertices (and $K_n = K_n^2$ as usual).  A well-studied problem has been determining the value of $\mc_r(K_n^k)$; however, this problem is still open for most values of $r$ and $k$.  

On the other hand, Gy\'arf\'as proved the following well-known results.

\begin{theorem}[Gy\'arf\'as \cite{Gy}]~\label{thm:Gy}
\begin{enumerate}
\item For all $n\geq r\geq 2$, $\mc_r(K_n)\geq \frac{n}{r-1}$.  This is best possible when $(r-1)^2$ divides $n$ and there exists an affine plane of order $r-1$.  
\item For all $n\geq r\geq 2$, $\mc_r(K_n^r)=n$.
\item For all $n\geq r\geq 4$, $\mc_{r}(K_n^{r-1})\geq \frac{(r-1)n}{r}$.  This is best possible for all such $r$ and $n$.  
\end{enumerate}
\end{theorem}

A natural question which has received attention lately has been to determine conditions under which a $k$-uniform hypergraph $G$ on $n$ vertices satisfies $\mc_r(G)= \mc_r(K_n^k)$ or $\mc_r(G)\geq (1-o(1))\mc_r(K_n^k)$, or if this is too restrictive, determining the value of $\mc_r(G)$ in terms of some natural parameters of $G$.   

Perhaps the first such result is due to F\"uredi \cite{F} who proved that for all graphs $G$ on $n$ vertices, $\mc_r(G)\geq \frac{n}{(r-1)\alpha(G)}$, which is best possible when an affine plane of order $r-1$ exists (see Section \ref{sec:Furedi} for more details).  But note that the value of $\mc_r(G)$ is far from $\mc_r(K_n)$ in this case.  In a sense that will be made precise in the coming pages, our paper is essentially a variant on F\"uredi's result using a different (but related) parameter in place of independence number for which we can guarantee that $\mc_r(G)$ is close $\mc_r(K_n)$.

Note that for $1\leq r\leq k$, $\mc_r(G)=n= \mc_r(K_n^k)$ if and only if the $r$-shadow of $G$ is complete\footnote{Indeed, if every $r$-set of $G$ is contained an edge, then since $\mc_r(K_n^r)=n$, we have $\mc_r(G)=n$.  Furthermore, if some $r$-set $\{x_1, \dots, x_r\}$ is not contained in an edge, then we can color the edges of $G$ with $r$-colors such that color $i$ is never used on $x_i$ and thus $\mc_r(G)<n$ (c.f. Observation \ref{mc2_upper}).}.  On the other hand, as first noted by Gy\'arf\'as and S\'ark\"ozy \cite{GSc}, when $r>k=2$ it is surprisingly possible for $\mc_r(G)= \mc_r(K_n)$ provided $G$ has large enough minimum degree.  See \cite{GSc}, \cite{R}, and \cite{FL} for the best known results on this minimum degree threshold in the case $k=2$, and \cite{BD2} for a precise result on the minimum codegree threshold in the case $r=k+1\geq 4$.  

For hypergraphs, Bennett, DeBiasio, Dudek, and English \cite{BDDE} proved that if $G$ is an $(r-1)$-uniform hypergraph on $n$ vertices with $e(G)\geq (1-o(1))\binom{n}{r-1}$, then $\mc_{r-1}(G)\geq (1-o(1))n$ and $\mc_{r}(G)\geq (\frac{r-1}{r}-o(1))n$.  

As for random graphs, it was independently determined in \cite{BD1}, \cite{DP2} that with high probability\footnote{An event is said to happen \emph{with high probability} or \emph{w.h.p.} if the probability that the event occurs tends to 1 as $n\to\infty$.}, 
$\mc_r(\gnp)\geq (1-o(1))\frac{n}{r-1}$ provided $p=\frac{\omega(1)}{n}$, and it was determined (using the result mentioned in the previous paragraph) in \cite{BDDE} that $\mc_r(\hnpr)\geq (1-o(1))n$ provided $p=\frac{\omega(1)}{n^{r-1}}$, and $\mc_{r}(H^{r-1}(n,p))\geq (1-o(1))\frac{(r-1)n}{r}$ provided $p=\frac{\omega(1)}{n^{r-2}}$.  All of these results for random graphs use the sparse regularity lemma and thus only provide weak bounds on the error terms.  Additionally, it was determined in \cite{DT} that for almost all Steiner triple systems $S$ on $n$ vertices, $\mc_3(S)=(1-o(1))n$.  In this case, there is an explicit bound on the error term, but their result is specific to 3 colors and 3-uniform hypergraphs in which every pair of vertices is contained in at least one edge.

In this paper, we study a common generalization which implies all of the results from the previous two paragraphs with more precise error terms.

\subsection{Relationship between monochromatic components and partite holes}

Given a hypergraph $G$, a \emph{$k$-partite hole of size $a$} is a collection of pairwise disjoint sets $X_1,\ldots, X_k \subseteq V(G)$  such that $|X_1|=\cdots=|X_k|=a$ and no edge $e\in E(G)$ satisfies $e\cap X_i\neq \emptyset$ for all $i\in [k].$ Define the \emph{$k$-partite hole number} $\alpha_k(G)$ to be the largest integer $a$ such that $G$ contains an $k$-partite hole of size $a$.

We consider the following general problem.

\begin{problem}\label{prob:gen}
Prove that for all integers $r,k\geq 2$, there exists $c_{r,k}, d_{r,k}>0$ such that for all $k$-uniform hypergraphs $G$ on $n$ vertices, if $\alpha_k(G)<c_{r,k} n$, then $$\mc_r(G)\geq \mc_r(K_n^k)-d_{r,k}\alpha_k(G).$$  Furthermore, determine the optimal values of $c_{r,k}, d_{r,k}$.  
\end{problem}

Note that because of the results mentioned above for graphs with large minimum degree, when $r\geq 3$ we can't necessarily get a $d'_{r,2}>0$ such that $\mc_r(G)\leq \mc_r(K_n)-d'_{r,2}\alpha_2(G)$ (because it is possible to have large minimum degree and large $\alpha_2(G)$).  However, when $r=k$, it is the case that $\mc_r(G)\leq \mc_r(K_n^r)-\alpha_r(G)$ (see Observation \ref{mc2_upper}).  

We solve Problem \ref{prob:gen} for all $1\leq r\leq k+1$, give the optimal values of $c_{r,k}, d_{r,k}$ in the case of $k=2=r$, give the optimal value of $c_{r,k}$ in the case $k=3=r$, and give reasonable estimates on $c_{r,k}, d_{r,k}$ in the other cases.  The formal statements are given below.

\begin{theorem}\label{thm:r2nu2}~
\begin{enumerate}
\item 
\begin{enumerate}
\item For all graphs $G$ on $n$ vertices, if $\alpha_2(G)< n/6$, then $\mc_2(G)\geq n-2\alpha_2(G)$. 
\item Furthermore, the bound on $\alpha_2(G)$ is best possible in the sense that there exists a graph on $n$ vertices with $\alpha_2(G)= n/6$ such that $\mc_2(G)=n/3$.
\end{enumerate}
\item For all graphs $G$ on $n$ vertices, $\mc_2(G)\leq n-\alpha_2(G)$.
\item For all integers $n$ and $a$ with $0\leq a\leq n/4$, there exists a graph on $n$ vertices with $\alpha_2(G)=a$ such that $\mc_2(G)\leq n-2a$.
\end{enumerate}
\end{theorem}

\begin{theorem}\label{thm:r3nu2}~
\begin{enumerate}
\item For all graphs $G$ on $n$ vertices, if $\alpha_2(G)\leq \frac{n}{3^9}$, then $\mc_3(G)\geq \frac{n}{2}-2\alpha_2(G)$. 
\item For all $0\leq a\leq n/2$, there exists a graph on $n$ vertices with $\alpha_2(G)=a$ such that $\mc_3(G)\leq \frac{n-a}{2}$.
\end{enumerate}
\end{theorem}

\begin{theorem}\label{thm:rrnur}
Let $r$ be an integer with $r\geq 3$.
\begin{enumerate}
\item There exists $c_r>0$ such that for all $r$-uniform hypergraphs $G$ on $n$ vertices, if $\alpha_r(G)<c_r n$, then $n-
\alpha_r(G)\geq \mc_r(G)\geq n-(r-1)\alpha_r(G)$.
\item For all $0 \le a \le n/(r+2)$, there exists a $r$-uniform hypergraph $G$ on $n$ vertices with $\alpha_r(G) = a$ such that  $\mc_r(G) \le n - 2\alpha_r(G)$.
\end{enumerate}
\end{theorem}

\begin{theorem}\label{thm:rrnur-1}
Let $r$ be an integer with $r\geq 4$.   
\begin{enumerate}
\item There exists $c_r>0$ such that for all $(r-1)$-uniform hypergraphs $G$ on $n$ vertices, if $\alpha_{r-1}(G)<c_r n$, then $\mc_{r}(G)\geq \frac{r-1}{r}n-\binom{r}{2} \alpha_{r-1}(G)$.
\item For all $0\leq a\leq n/(r-1)$, there exists an $(r-1)$-uniform hypergraph on $n$ vertices with $\alpha_{r-1}(G)=a$ such that $\mc_{r}(G)\leq \frac{r-1}{r}(n-a)$.
\end{enumerate}
\end{theorem}

The first completely open case of Problem \ref{prob:gen} is when $k=2$ and $r\geq 4$.  In this case, we give the following modest lower bound (see Conjecture \ref{con:k=2} and the preceding discussion for more details about this open case).

\begin{proposition}\label{propweak}
For all $r\geq 4$, there exists $c_r >0$ such that for all graphs $G$ on $n$ vertices, if $\alpha_2(G)<c_r n$, then $\mc_r(G)\geq \frac{n-\alpha_2(G)}{r}$.  
\end{proposition}

It would also be interesting to improve the bounds in Theorem \ref{thm:rrnur}. In particular, we have (in the context of Problem \ref{prob:gen}) that $2\le d_{r,r}\le (r-1)$.

\subsection{Corollaries}

As mentioned earlier, there have been a number of results showing that $\mc_r(H^k(n,p))=(1-o(1))\mc_r(K_n^k)$ where $H^k(n,p)$ is the binomial random $k$-uniform hypergraph.  However, those results have all used the sparse regularity lemma and thus there are no reasonable estimates on the error terms.  Since the value of $\alpha_k(H^k(n,p))$ is easy to estimate, for all values of $k$ and $r$ for which Theorems \ref{thm:r2nu2} - \ref{thm:rrnur-1} hold we automatically recover $\mc_r(H^k(n,p))=(1-o(1))\mc_r(K_n^k)$ with very good estimates on the error terms.  

\begin{corollary}
For all $r\ge 2$ and $p = \frac{d}{n^{r-1}}$ with $d\to \infty$, we have that with high probability,
\[\mc_r(H^r(n,p)) = n-\Theta\of{\of{\frac{\log d}{d}}^{\frac{1}{r-1}}n}.\]
Additionally,
\[\mc_3(H^2(n,p))\ge \frac{n}{2} - O\of{\frac{\log d}{d} n}\] 
and for all $r\ge 3$,
\[\mc_{r+1}(H^{r}(n,p)) \ge \frac{r}{r+1}n - O\of{\of{\frac{\log d}{d}}^{\frac{1}{r-1}}n}\]
\end{corollary}

\begin{proof}
The statements above follow from Theorems \ref{thm:r2nu2} - \ref{thm:rrnur-1} and the fact that for $p$ as in the statement, w.h.p., $\alpha_r(H^r(n,p)) = \Theta\of{\of{\frac{\log d}{d}}^{\frac{1}{r-1}}n}$ where the $\Theta$-notation is suppressing a multiplicative constant which may depend on $r$. The upper bound can be shown using a standard first moment calculation. The lower bound follows by taking an independent set of size $\Theta\of{\of{\frac{\log d}{d}}^{\frac{1}{r-1}}n}$ and partitioning it into $r$ equal sized sets (see Observation \ref{obs:alpha}). Independent sets of this size are known to exist (see e.g. \cite{KS}).
\end{proof}

See Observation \ref{obs:rg-not-all} and Problem \ref{prob:rhg-not-full} for a discussion about the upper bounds in the second and third statements.

Let $\mathcal{S}_n$ be the family of all Steiner triple systems on $n$ vertices.  DeBiasio and Tait \cite{DT} proved that for all 3-uniform hypergraphs $G$ on $n$ vertices in which every pair of vertices is contained in at least one edge, $\mc_3(G)\geq n-2\alpha_3(G)$ (note that Theorem \ref{thm:rrnur}(i) is stronger in the sense that there is no requirement that every pair of vertices is contained in at least one edge).  They used this to prove that for all $S\in \mathcal{S}_n$, $\mc_3(S)\geq 2n/3+1$ and there exists $\delta>0$ such that for almost all $S\in \mathcal{S}_n$, $\mc_3(S)\geq n-n^{1-\delta}$.  This latter result was proved by showing that for almost all $S\in \mathcal{S}_n$, $\alpha_3(S)\leq n^{1-\delta}$.  
Gy\'arf\'as \cite{Gy2} proved in particular that for all $S\in \mathcal{S}_n$, $\mc_4(S)\geq \frac{n}{3}$ (and this is best possible for infinitely many $n$).  Using the fact (from \cite{DT}) that for almost all $S\in \mathcal{S}_n$, $\alpha_3(S)\leq n^{1-\delta}$, we obtain the following corollary of Theorem \ref{thm:rrnur-1}(i) (with $r=4$).

\begin{corollary}~
There exists $\delta>0$ such that for almost all $S\in \mathcal{S}_n$, $\mc_4(S)\geq \frac{3n}{4}-O(n^{1-\delta})$.
\end{corollary}

\subsection{Expansion}

Let $G$ be a $k$-uniform hypergraph $G$ on $n$ vertices and let $S_1, \dots, S_{k-1}\subseteq V(G)$.  Define $N(S_1,\dots, S_{k-1})=\{v: \{v_1, \dots, v_{k-1}, v\}\in E(G), v_i\in S_i \text{ for all } i\in [k-1]\}$ and $N^+(S_1,\dots, S_{k-1})=\{v\in V(G)\setminus (S_1\cup\dots\cup S_{k-1}): \{v_1, \dots, v_{k-1}, v\}\in E(G), v_i\in S_i \text{ for all } i\in [k-1]\}$.

We say that a $k$-uniform hypergraph $G$ on $n$ vertices is a \emph{$(p,q)$-expander} if for all sets $S_1, \dots, S_{k-1}\subseteq V(G)$ with $|S_i|>p$ for all $i\in [k-1]$, we have $|N(S_1, \dots, S_{k-1})|\geq q$.  

We say that a $k$-uniform hypergraph $G$ on $n$ vertices is a \emph{$(p,q)$-outer-expander} if for all \emph{disjoint} sets $S_1, \dots, S_{k-1}\subseteq V(G)$ with $|S_i|>p$ for all $i\in [k-1]$, we have $|N^+(S_1, \dots, S_{k-1})|+|S_1\cup \dots \cup S_{k-1}|\geq q$.

Given a hypergraph $G$ and an integer $r\geq 2$, let $\halpha_r(G)$ be the largest integer $a$ such that there exists (not-necessarily disjoint) sets $V_1, \dots, V_r$ with $|V_i|=a$ for all $i\in [r]$ such that there are no edges $e$ such that $e\cap V_i\neq \emptyset$ for all $i\in [r]$.

We first make an observation regarding the relationship between $\alpha_k(G)$, $\halpha_k(G)$, $\alpha(G)$.  One takeaway from this observation is that it would make very little difference in our results if we considered bounding $\halpha_k(G)$ instead of $\alpha_k(G)$.  However, it is possible for $\alpha(G)$ to be small and $\alpha_k(G)$ to be large (the disjoint union of cliques of order $n/k$ for instance), so it makes a big difference if we were to consider bounding $\alpha(G)$ instead of $\alpha_k(G)$.  

\begin{observation}\label{obs:alpha}
For all $k$-uniform hypergraphs $G$, $$\floor{\frac{\alpha(G)}{k}}\leq \floor{\frac{\halpha_k(G)}{k}}\leq \alpha_k(G)\leq \halpha_k(G).$$
\end{observation}

\begin{proof}
First note that if $S$ is an independent set, then by letting $V_1=\dots=V_k=S$, we have $\halpha_k(G)\geq |S|$.  So $\alpha(G)\leq \halpha_k(G)$.  Also we clearly have $\alpha_k(G)\leq \halpha_k(G)$ since $\halpha_k$ is computed over a strictly larger domain than $\alpha_k$ (all collections of sets vs. all collections of disjoint sets).  

Now let $V_1, \dots, V_k\subseteq V(G)$ (not-necessarily-disjoint) be sets such that  $|V_1|=\dots=|V_k|$ and there are no edges $e$ such that $e\cap X_i\neq\emptyset$.  For all $i\in [k]$, there exists $V_i'\subseteq V_i$ with $|V_i'|\geq \floor{\frac{|V_i|}{k}}$ such that $V_i'\cap V_j'=\emptyset$ for all distinct $i,j\in [k]$.  Since there are no edges which intersect all of $V_1, \dots, V_k$, there are no edges which intersect all of $V_1', \dots, V_k'$ and thus we have $\alpha_k(G)\geq \floor{\frac{\halpha_k(G)}{k}}$.
\end{proof}

We now make an observation which provides the relationship between small $k$-partite holes and expansion.

\begin{observation}
Let $G=(V,E)$ be a $k$-uniform hypergraph on $n$ vertices.
\begin{enumerate}
\item $G$ is a $(p,n-p)$-expander if and only if $\halpha_k(G)\leq p$.
\item $G$ is a $(p,n-p)$-outer-expander if and only if $\alpha_k(G)\leq p$.
\end{enumerate}
\end{observation}

\begin{proof}
(i) Let $S_1, \dots, S_{k-1}\subseteq V$ with $|S_i|>p$ for all $i\in [k-1]$.  If $|N(S_1, \dots, S_{k-1})|<n-p$, then $|V\setminus N(S_1, \dots, S_{k-1})|>p$ and there are no edges touching all of $S_1, \dots, S_{k-1}, V\setminus N(S_1, \dots, S_{k-1})$ which implies $\halpha_k(G)>p$.  

Now suppose $G$ is a $(p,n-p)$-expander and let $S_1, \dots, S_{k}\subseteq V$ with $|S_i|>p$ for all $i\in [k]$.  Since $|N(S_1, \dots, S_{k-1})|\geq n-p$, we have $S_k\cap N(S_1, \dots, S_{k-1})\neq \emptyset$; i.e. there is an edge which touches all of $S_1, \dots, S_k$ and thus $\halpha_k(G)\leq p$.

(ii) Let $S_1, \dots, S_{k-1}\subseteq V$ be disjoint sets with $|S_i|>p$ for all $i\in [k-1]$.  If $|N^+(S_1, \dots, S_{k-1})|+|S_1\cup \dots \cup S_{k-1}|<n-p$, then $|V\setminus (N^+(S_1, \dots, S_{k-1})\cup (S_1\cup\dots\cup S_{k-1}))|>p$ and there are no edges touching all of $S_1, \dots, S_{k-1}, V\setminus (N^+(S_1, \dots, S_{k-1})\cup (S_1\cup\dots\cup S_{k-1}))$ which implies $\alpha_k(G)>p$.  

Now suppose $G$ is a $(p,n-p)$-outer-expander and let $S_1, \dots, S_{k}\subseteq V$ be disjoint sets with $|S_i|>p$ for all $i\in [k]$.  Since $|N^+(S_1, \dots, S_{k-1})|+|S_1\cup \dots \cup S_{k-1}|\geq n-p$, we have $S_k\cap N^+(S_1, \dots, S_{k-1})\neq \emptyset$; i.e. there is an edge which touches all of $S_1, \dots, S_k$ and thus $\alpha_k(G)\leq p$.
\end{proof}

\subsection{Outline of Paper}
In Section \ref{sec:duality}, we discuss a reformulation of our problem in the dual language of $r$-partite $r$-uniform hypergraphs which we will work with for the remainder of the paper. In Section \ref{sec:examples} we provide examples which show the tightness of our results. In particular, this section contains proofs of Theorem \ref{thm:r2nu2}(i)(b), (ii), (iii), Theorem \ref{thm:r3nu2} (ii), Theorem \ref{thm:rrnur} (ii) and Theorem \ref{thm:rrnur-1} (ii). 
In Section \ref{sec:maindual} we prove Theorem \ref{thm:r2nu2}(i)(a), Theorem \ref{thm:r3nu2}(i), Theorem \ref{thm:rrnur}(i), Theorem \ref{thm:rrnur-1}(i), and Proposition \ref{propweak}.

\section{Duality}\label{sec:duality}

Throughout the rest of the paper we will be talking about multi-hypergraphs and we will always assume that all of the edges are distinguishable (and more generally, we assume that all of the elements in a multi-set are distinguishable).  This means, for example, that if an edge has multiplicity $5$, we can partition those five edges into two disjoint sets of say $3$ and $2$ edges respectively.

Let $r,k\geq 2$ be integers.  Given an $r$-partite $r$-uniform multi-hypergraph $H$ and multisets of edges $E_1, \dots, E_k$, we say that $E_1, \dots, E_k$ is \emph{cross-intersecting} if there exists $e_i\in E_i$ for all $i\in [k]$ such that $e_1\cap \dots \cap e_k\neq \emptyset$.  Furthermore, if $S\subseteq V(H)$, we say that $E_1, \dots, E_k$ is \emph{cross-intersecting in $S$} if there exists $e_i\in E_i$ for all $i\in [k]$ such that $S\cap e_1\cap \dots \cap e_k\neq \emptyset$.

Let $\nu_k(H)$ be the largest integer $m$ such that there exists multisets of edges $E_1, \dots, E_k$ with $|E_i|=m$ for all $i\in [k]$ and $E_i\cap E_j=\emptyset$ for all distinct $i,j\in [k]$ such that $E_1, \dots, E_k$ is not cross-intersecting; i.e. $e_1\cap e_2\cap \dots \cap e_k=\emptyset$ for all $e_1\in E_1$, $e_2\in E_2$, $\dots$, $e_k\in E_k$.

\subsection{Monochromatic components and $k$-partite holes}

The following observation precisely describes what we mean by ``duality.''

\begin{observation}[Duality]\label{dual:r}
Let $n\geq 1$, $r,k\geq 2$, and $s,t\geq 0$.  The following are equivalent:
\begin{enumerate}
\item Let $G$ be a $k$-uniform hypergraph on $n$ vertices.  If $\alpha_k(G)\leq s$, then $\mc_r(G)\geq t$
\item Let $H$ be an $r$-partite $r$-uniform multi-hypergraph with $n$ edges.  If $\nu_k(H)\leq s$, then $\Delta(H)\geq t$.
\end{enumerate}
\end{observation}

\begin{proof}
First, suppose (i). Let $H$ be an $r$-partite $r$-uniform multi-hypergraph with $\nu_k(H)\leq s$.  Let the parts of $H$ be labeled $C^i=\{C^i_1, \dots, C^i_{k_i}\}$ for all $i\in [r]$.   Let $G$ be an $r$-edge colored $k$-uniform hypergraph with $V(G)=E(H)$ where $\{e_1,\dots, e_k\}\in E(G)$ of color $i$ if and only if $(e_1\cap \dots\cap e_k)\cap C^i\neq \emptyset$ (in $H$); note that an edge of $G$ may receive multiple colors.  If $E_1,E_2, \dots, E_k\subseteq V(G)$ such that $e_G(E_1, E_2, \dots, E_k)=0$, then $\bigcap_{i\in [k]}(\bigcup_{e\in E_i} e)=\emptyset$ (in $H$).  So we have $\alpha_k(G)\leq \nu_k(H)\leq s$, which by the assumption implies that $G$ has a monochromatic component $C^i_j$ of order at least $t$.  Since $C^i_j$ has at least $t$ vertices, this implies that the vertex $C^i_j$ in $H$ has degree at least $t$.

Next, suppose (ii).  Let $G$ be a $k$-uniform hypergraph with $\alpha_k(G)\leq s$. Suppose we are given an $r$-coloring $\phi$ of $G$ and let $C^i_1, \dots, C^i_{k_i}$ be the components of $G$ of color $i$ for all $i\in [r]$.  Let $H$ be an $r$-partite $r$-uniform multi-hypergraph with parts $C^i=\{C^i_1, \dots, C^i_{k_i}\}$ for all $i\in [r]$ where $\{C^1_{j_1},\dots, C^r_{j_r}\} \in E(H)$ is an edge of multiplicity $m$ if and only if $|\bigcap_{i\in [r]} V(C^i_{j_i})|=m$ (in $G$); note that ``an edge of multiplicity 0'' just means a non-edge.  Note that $|V(G)|=|E(H)|$ (in fact, we essentially have $E(H)=V(G)$) since every vertex in $G$ is in exactly one component of each color.  If there exists $E_1, E_2, \dots, E_k\subseteq E(H)$ such that $E_i\cap E_j=\emptyset$ for all distinct $i,j\in [k]$ and $e_1\cap e_2\cap \dots \cap e_k=\emptyset$ for all $e_1\in E_1$, $e_2\in E_2$, $\dots$, $e_k\in E_k$, then $e_G(E_1, E_2, \dots, E_k)=0$ because any such edge intersecting all of $E_1, \dots, E_k$ (in $G$) would violate $e_1\cap e_2\cap \dots \cap e_k=\emptyset$ (in $H$).  So we have $\nu_k(H)\leq \alpha_k(G)\leq s$ which by the assumption implies $\Delta(H)\geq t$.  Without loss of generality, suppose $d_H(C^1_1)=\Delta(H)\geq t$ which means $C^1_1$ is a component of color $1$ in $G$ with at least $t$ vertices.
\end{proof}

\subsection{Monochromatic components and independence number}\label{sec:Furedi}

For expository reasons and as a comparison to the result in the last subsection, we describe F\"uredi's classic example of the use of duality. 

For a hypergraph $H$, let $\tau(H)$ denote the vertex cover number, let $\nu(H)$ denote the matching number and let $\tau^*$ and $\nu^*$ denote the respective fractional versions.  Ryser conjectured that for every $r$-partite (multi)hypergraph $H$, $\tau(H)\leq (r-1)\nu(H)$.  F\"uredi \cite{F} proved a fractional version; that is, for every $r$-partite (multi)hypergraph $H$, $\tau^*(H)\leq (r-1)\nu(H)$.  Since $$\frac{n}{\Delta(H)}\leq \nu^*(H)=\tau^*(H)\leq (r-1)\nu(H),$$ it follows that for every $r$-partite (multi)hypergraph $H$ with $n$ edges, $\Delta(H)\geq \frac{n}{(r-1)\nu(H)}$.  
In the dual language, this says for every graph $G$ on $n$ vertices, $\mc_r(G)\geq \frac{n}{(r-1)\alpha(G)}$.

\section{Examples}\label{sec:examples}

The first example provides the upper bound in Theorem \ref{thm:r2nu2}.(ii).  

\begin{observation}\label{mc2_upper}Let $2\leq r\leq k$.  
For all $k$-uniform hypergraphs $G$ on $n$ vertices, $\mc_r(G)\leq n-\alpha_r(G)$.
\end{observation}

\begin{proof}
Let $X_1, \dots, X_r$ be disjoint sets which witness the value of $\alpha_r(G)$; that is, disjoint sets with $|X_i|=\alpha_r(G)$ for all $i\in [r]$ such that $e(X_1, \dots, X_r)=\emptyset$.  For all $i\in [r]$, color all edges not incident with $X_i$ with color $i$ (so edges may receive many colors).  Since every edge misses some $X_i$, every edge receives at least one color.  So every component of color $i$ avoids $X_i$ and thus has order at most $n-\alpha_r(G)$.
\end{proof}

The next example provides the proof of Theorem \ref{thm:r2nu2}.(iii) and Theorem \ref{thm:rrnur}.(ii).

\begin{example}\label{mcr_upper}
For all integers $n\geq r\geq 2$ and $0\leq a\leq n/(r+2)$, there exists a $r$-uniform hypergraph $G$ on $n$ vertices with $\alpha_r(G)=a$ such that $\mc_r(G)= n-2\alpha_r(G)$.
\end{example}

\begin{proof}
Let $V$ be a set of order $n$ and let $\{V_0,V_1, \dots, V_r,V_{r+1}\}$ be a partition of $V$ with $|V_0|=n-(r+1)a$, $|V_1|=\dots=|V_{r}|=|V_{r+1}|=a$.  Let $G$ be an $r$-uniform hypergraph on $V$ where $e$ is an edge if and only if $|e|=r$ and $e\subseteq V_0\cup\bigcup_{i\in [r]\setminus\{j\}}V_i$ for some $j\in [r]$, or $e\subseteq V_j\cup V_{r+1}$ for some $j\in[r]$.  If $e\subseteq V_j\cup V_{r+1}$ or $e\subseteq V_0\cup\bigcup_{i\in [r]\setminus\{j\}}V_i$, color $e$ with $j$.  Note that for all $j\in [r]$, there is a monochromatic component of color $j$ containing $V_j\cup V_{r+1}$ and a disjoint monochromatic component of color $j$ containing $V_0\cup\bigcup_{i\in [r]\setminus\{j\}}V_i$ and thus the largest monochromatic component has order $\max\{n-2a,2a\}=n-2a$ as desired.  It is straightforward, albeit tedious, to check that $\alpha_r(G)=a$.
\end{proof}

For expository reasons, we give the same example as above in the dual language. 

\begin{example}\label{mcr_upper_dual}
For all integers $n\geq r\geq 2$ and $a\geq 0$ with $a\leq n/(r+2)$, there exists an $r$-uniform hypergraph $H$ on $n$ vertices with $\nu_r(H)=a$ such that $\Delta(H)= n-2\nu_r(H)$.
\end{example}

\begin{proof}
Let $H$ be an $r$-partite hypergraph with two vertices $u_i, v_i$ in each part.  Let $\{v_1, \dots, v_r\}$ be an edge of multiplicity $a$.  For all $i\in [r]$, let $\{u_1, \dots, u_{i-1}, v_i, u_{i+1}, \dots, u_r\}$ be an edge of multiplicity $a$.  Finally, let $\{u_1, \dots, u_r\}$ be an edge of multiplicity $n-(r+1)a$.  
Note that every vertex in $\{u_1, \dots, u_r\}$ has degree $n-2a$ and every vertex in $\{v_1, \dots, v_n\}$ has degree $2a$, so $\Delta(H)= \max\{n-2a,2a\}=n-2a$ as desired.  It is again straightforward to see that $\nu_r(H)=a$.
\end{proof}

The next example provides the proof of Theorem \ref{thm:rrnur-1}(ii).

\begin{example}For all $r\geq 2$ and $1\leq a\leq n/k$, there exists a $k$-uniform hypergraph $G$ on $n$ vertices with $\alpha_k(G)=a$ such that $\mc_r(G)\leq \mc_r(K_{n-a}^k)<\mc_r(K_n^k)$.
\end{example}

\begin{proof}
Let $G$ be a complete $k$-uniform hypergraph on $n-a$ vertices together with $a$ isolated vertices.  We have $\mc_r(G)=\mc_r(K_{n-a}^k)<\mc_r(K_{n}^a)$.  
\end{proof}

The next example provides the proof of Theorem \ref{thm:r2nu2}.(i)(b).  For instance when $s=3$, $t=4$, this gives an example of a 2-colored graph $G$ with $\alpha_2(G)=\frac{n}{6}$ where the largest monochromatic component has order $n/3$. 

\begin{example}\label{ex:grid}
Let $n\geq t\geq s$ be positive integers such that $st$ divides $n$.  The 2-colored $(s,t)$-grid, denoted $G_2(s,t)$, is the graph obtained by partitioning $[n]$ into $st$ sets $A_{11}, \dots, A_{1t}$, $A_{21}, \dots, A_{2t}$, $\dots$, $A_{s1}, \dots, A_{st}$, each of order $\frac{n}{st}$.  For all $i\in [s]$, let $A_{i1}\cup \dots \cup A_{it}$ be a red clique and let $A_{1i}\cup \dots \cup A_{si}$ be a blue clique (edges in the intersection can be colored with both colors say).  The largest monochromatic component has order at most $n/s$ and $\alpha_2(G_2(s,t))=\frac{\ceiling{s/2}\floor{t/2}}{st}n$.
\end{example}

\begin{proof}
The fact that the largest monochromatic component has order at most $n/s$ is evident.  To see that $\alpha_2(G_2(s,t))=\frac{\ceiling{s/2}\floor{t/2}}{st}n$, let $X, Y\subseteq V(G_2(s,t))$ be maximal disjoint sets witnessing the value of $\alpha_2(G)$; i.e.\ $\min\{|X|, |Y|\}=\alpha_2(G)$ and $e(X,Y)=0$.  By the maximality of $X,Y$  and the structure of $G_2(s,t)$, we have that if $X\cap A_{ij}\neq \emptyset$ then $X\cap A_{ij}=A_{ij}$ and likewise $Y\cap A_{ij}\neq \emptyset$ implies $Y\cap A_{ij}=A_{ij}$.  Let $I=\{i\in [s]: X\cap A_{ij}\neq \emptyset$ for some $j\in [t]\}$ and $J=\{j\in [t]: X\cap A_{ij}\neq \emptyset$ for some $i\in [s]\}$.  This implies that if $Y\cap A_{ij}\neq \emptyset$, then $i\in [s]\setminus I$ and $j\in [t]\setminus J$. So we have $|X|=\frac{|I||J|n}{st}$ and $|Y|=\frac{(s-|I|)(t-|J|)n}{st}$ and $\alpha_2(G)=\min\{|X|, |Y|\}$ is maximized when $|I|=\ceiling{s/2}$ and $|J|=\floor{t/2}$ (equivalently, $|I|=\floor{s/2}$ and $|J|=\ceiling{t/2}$). 
\end{proof}

For random graphs $G(n,p)$, it was shown in \cite{BD1} and \cite{DP2} that $\mc_r(\gnp)\geq \left(\frac{1}{r-1}-o(1)\right)n$, and thus whenever an affine plane of order $r-1$ exists, 
we have $\mc_r(G(n,p))= (1-o(1))\mc_r(K_n)$.  The following observation shows that we cannot hope for $\mc_r(G(n,p))= \mc_r(K_n)$ (unless $p$ is close to 1 and thus the minimum degree of $G(n,p)$ is close to $n$; see \cite{FL}).

\begin{observation}\label{obs:rg-not-all}
Let $r$ and $C$ be integers with $r\geq 2$ and $C\geq 1$.  
If an affine plane of order $r-1$ exists, then for sufficiently small $\frac{\omega(1)}{n}=p<1$ we have w.h.p., $\mc_r(\gnp)\leq \frac{n}{r-1}-C.$
\end{observation}

\begin{proof}
We choose $p$ small enough such that with high probability, $G(n,p)$ has an independent set $A$ with $Cr$ vertices such that $|\bigcup_{v\in A}N(v)|\leq \frac{n-Cr}{(r-1)^2}$.  Partition $A$ into $r$ sets $\{A_1, \dots, A_r\}$ each of order $C$ and partition the vertices of $G$ into sets of size $\frac{n-Cr}{(r-1)^2}$ with one of those sets containing $\bigcup_{v\in A}N(v)$ and color the edges of $G-A$ according to the affine plane coloring.  Now color all edges incident with $A_i$ with color $i$ for all $i\in [r]$.  So every component of color $i$ has order at most $\frac{n-Cr}{r-1}+C\leq \frac{n}{r-1}-C$.
\end{proof}

For random hypergraphs $H^k(n,p)$ it is not clear to us whether for $r>k$ we can hope for $\mc_r(H^k(n,p))= \mc_r(K_n^k)$ whenever we have $\mc_r(H^k(n,p))= (1-o(1))\mc_r(K_n^k)$.

\begin{problem}\label{prob:rhg-not-full}
Let $r$ and $C$ be integers with $r\geq 4$ and $C\geq 1$.
Prove that for sufficiently small $\frac{\omega(1)}{n^{r-2}}=p<1$ we have w.h.p., $$\left(\frac{r-1}{r}-o(1)\right)n\leq \mc_r(H^{r-1}(n,p))\leq \frac{(r-1)n}{r}-C.$$
\end{problem}

\section{Main results in the dual language}\label{sec:maindual}

All of the results of this section are of the type ``For all $k, r$ there exists $c_{k,r}, d_{k,r}$ such that if $G$ is a $k$-uniform hypergraph on $n$ vertices with $\alpha_k(G)<c_{k,r} n$, then $\mc_r(G)\geq \mc_r(K_n^k)-d_{k,r}\alpha_k(G)$;'' however we prove these statements in the equivalent dual form ``For all $k, r$ there exists $c_{k,r}, d_{k,r}>0$ such that if $H$ is an $r$-partite $r$-uniform multihypergraph with $n$ edges and $\nu_k(H)<c_{k,r}$, then $\Delta(H)\geq \mc_r(K_n^k)-d_{k,r}\nu_k(G)$.''

\begin{theorem}[Dual of Theorem \ref{thm:r2nu2}(i)(a)]\label{dual:r2nu2}
Let $H$ be a bipartite multigraph with $n$ edges.  If $\nu_2(H)<n/6$, then $\Delta(H)\geq n-2\nu_2(H)$.
\end{theorem}

\begin{theorem}[Dual of Theorem \ref{thm:r3nu2}(i)]\label{dual:r3nu2}
Let $H$ be an $3$-partite $3$-uniform multi-hypergraph with $n$ edges.  If $\nu_{2}(H)\leq \frac{n}{3^9}$, then $\Delta(H)\geq \frac{n}{2}-2\nu_2(H)$.
\end{theorem}

\begin{theorem}[Dual of Theorem \ref{thm:rrnur}(i)]\label{dual:rrnur}
Let $r\geq 3$ and let $H$ be an $r$-partite $r$-uniform hypergraph with $n$ edges.  If $\nu_r(H)\leq \frac{n}{3^{\binom{r+1}{2}+r}}$, then $\Delta(H)\geq n-(r-1)\nu_r(H)$.
\end{theorem}

\begin{theorem}[Dual of Theorem \ref{thm:rrnur-1}(i)]\label{dual:rrnur-1}
Let $r\geq 3$ and let $H$ be an $r$-partite $r$-uniform hypergraph with $n$ edges.  If $\nu_{r-1}(H)\leq \frac{n}{3^{\binom{r+1}{2}+r}}$, then $\Delta(H)\geq \frac{(r-1)n}{r}-\binom{r}{2}\nu_{r-1}(H)$.
\end{theorem}

\subsection{General lemmas}

In this section we collect a number of general lemmas.  We begin with an elementary lemma that will be used throughout the proofs in this section.

\begin{lemma}\label{ballbin_precise}
Let $\nu, \ell, a_1, \dots, a_\ell$ be positive integers.  Let $H$ be an $r$-partite $r$-uniform multi-hypergraph with parts $V_1,\ldots, V_r$ and let $F_1, \dots, F_\ell \subseteq E(H)$ such that $|F_j|\geq 2a_j\nu+1$ for all $j\in [2,\ell]$ and $|F_1|\geq 3a_1\nu+1$. For all $i\in [r]$, either
\begin{enumerate}[label=\emph{(B\arabic*)}, ref=(B\arabic*)]
\item\label{b1} there exists $u\in V_i$ such that $u$ is incident with 
at least $|F_j|-a_j\nu$ edges of $F_j$ for all $j\in [\ell]$, or 

\item\label{b2} there exists a subset $F_1'\subseteq F_1$ with $|F_1'|\geq a_1\nu+1$ and a subset $F_j'\subseteq F_j$ for some $j\in [2,\ell]$ with $|F_j'|\geq a_j\nu+1$ such that $F_1',F_2 \dots, F_{j-1}, F_j', F_{j+1}, \dots, F_{\ell}$ is not cross-intersecting in $V_i$.
\end{enumerate}
\end{lemma}

\begin{proof}
Let $i\in [r]$.  If there exists $u\in V_i$ such that $u$ is incident with at least $|F_1|-a_1\nu$ edges from $F_1$, then either at least $a_j\nu+1$ edges of $F_j$ intersect $V_{i}-u$ for some $j\in [2,\ell]$ and \ref{b2} is satisfied, or $V_i-u$ is incident with at most $a_j\nu$ edges of $F_j$ for all $j\in [2,\ell]$ and thus $u$ is incident with at least $|F_j|-a_j\nu$ edges of $F_j$ for all $j\in [\ell]$ and \ref{b1} is satisfied.

So suppose that every $u\in V_i$ is incident with at most $|F_1|-a_1\nu-1$ edges of $F_1$.  Let $V_i^1\subseteq V_i$ be a minimal set of vertices incident with at least $a_1\nu+1$ edges of $F_1$.  By minimality, and the fact that every $u\in V_i$ is incident with at most $|F_1|-a_1\nu-1$ edges of $F_1$, we have that both $V_i^1$ and $V_i^2:=V_i\setminus V_i^1$ are incident with at least $a_1\nu+1$ edges of $F_1$.  Now either $V_i^1$ or $V_i^2$ is incident with at least $|F_2|-a_2\nu\geq a_2\nu+1$ edges of $F_2$, and either way \ref{b2} is satisfied. 
\end{proof}

A simpler version of the above lemma which suffices whenever we don't care about the exact bounds is as follows.

\begin{lemma}\label{ballbin}
Let $\nu, \ell, a_1, \dots, a_{\ell}$ be positive integers.  Let $H$ be an $r$-partite $r$-uniform multi-hypergraph with parts $V_1,\dots, V_r$, and let $F_1, \dots, F_\ell \subseteq E(H)$ such that $|F_j|\geq 3a_j\nu+1$ for all $j\in [\ell]$. 
For all $i\in [r]$, either
\begin{enumerate}[label=\emph{(B\arabic*$^\prime$)}, ref=(B\arabic*$^\prime$)]
\item\label{b1'} there is a vertex in $V_i$ which is incident with at least $|F_j|-a_j\nu$ edges of $F_j$ for all $j\in [\ell]$, or
\item\label{b2'} there exists $F_j'\subseteq F_j$ with $|F_j'|\geq a_j\nu+1$ for all $j\in [\ell]$ such that $F_1', \dots, F_\ell'$ is not cross-intersecting in $V_i$.
\end{enumerate}
\end{lemma}

\begin{observation}\label{obs:s<t}
Let $2\leq s< t\leq r$ and let $H$ be an $r$-partite $r$-uniform multi-hypergraph on $n$ edges.  If $\nu_t(H)\leq \frac{n}{t}-1$, then $\nu_s(H)\leq \nu_t(H)$.
\end{observation}

\begin{proof}
Suppose $\nu_t(H)\leq \frac{n}{t}-1$ and suppose for contradiction that $\nu_s(H)>\nu_t(H)$.  So there exists disjoint sets $E_1, \dots, E_s$ with $|E_i|= \nu_t(H)+1$ for all $i\in [s]$ such that $E_1, \dots, E_s$ is not cross-intersecting.  Since $\nu_t(H)\leq \frac{n}{t}-1$, we have $|E_i|\leq \frac{n}{t}$ for all $i\in [s]$ and thus $|E(H)\setminus (E_1\cup \dots \cup E_s)|\geq n-s\frac{n}{t}= (t-s)\frac{n}{t}$ and thus there is a partition of $E(H)\setminus (E_1\cup \dots \cup E_s)$ into $t-s$ sets $E_{s+1}, \dots, E_t$ each of order greater than $\nu_t(H)$ such that $E_1, \dots, E_s, E_{s+1}, \dots, E_t$ is not cross-intersecting.  
\end{proof}

We now show that if $H$ is an $r$-partite $r$-uniform multi-hypergraph on $n$ edges with $\nu_2(H)$ small enough in terms of $n$ and $r$, then there must be a vertex of fairly large degree.  Note that by Observation \ref{obs:s<t}, if we have a bound on $\nu_s(H)$ for $3\leq s\leq r$, this gives us a bound on $\nu_2(H)$ and that is why the following lemma (which only really uses a bound on $\nu_2(H)$) applies in all cases.

\begin{lemma}\label{decentdegree}
Let $\Delta, r, s$ be positive integers with $2\leq s\leq r$ and let $H$ be an $r$-partite $r$-uniform multi-hypergraph with $n$ edges.  If $\nu_s(H)\leq \frac{n}{3^r\Delta}$, then $\Delta(H)\geq  \Delta\cdot \nu_s(H)+1$.
\end{lemma}

\begin{proof}
Suppose $\nu_s(H)\leq \frac{n}{3^r\Delta}$ and suppose for contradiction that $\Delta(H)\leq  \Delta \nu_s(H)$.  Note that by Observation \ref{obs:s<t} we have $\nu:=\nu_2\leq \nu_s\leq \frac{n}{3^r\Delta}$.  

Let $V_1, \dots, V_r$ be the parts of $H$.  Let $V_1'\subseteq V_1$ be a minimum set of vertices incident with at least $3^{r-1}\Delta\nu+1$ edges.  By minimality, we have $$3^{r-1}\Delta \nu+1\leq e(V_1')\leq 3^{r-1}\Delta\nu+\Delta\nu$$ and consequently, since $\nu<\frac{n}{3^r\Delta}$, we have $$e(V_1\setminus V_1')= n-e(V_1')\geq n-3^{r-1}\Delta\nu-\Delta\nu> 3^{r-1}\Delta \nu.$$  Let $F_1^1$ and $F_2^1$ be the sets of edges incident with $V_1'$ and $V_1\setminus V_1'$ respectively.  Now we apply Lemma \ref{ballbin} (with $a_1=a_2=3^{r-2}\Delta$ and $i=2$), and since we are assuming $\Delta(H)\leq  \Delta \nu$, \ref{b2'} must happen.  Now we have sets $F_1^2\subseteq F_1^1$ and $F_2^2\subseteq F_2^1$ such that $|F_1^2|, |F_2^2|\geq 3^{r-2}\Delta\nu+1$ and $F_1^2$ and $F_2^2$ 
are not cross-intersecting in
$V_1\cup V_2$.  Now we repeatedly apply Lemma \ref{ballbin} until we have sets $F_1^{r-1}$ and $F_{2}^{r-1}$ with $|F_{1}^{r-1}|, |F_{2}^{r-1}|\geq 3\Delta\nu+1$ and $F_{1}^{r-1}$ and $F_{2}^{r-1}$ 
are not cross-intersecting in
$V_1\cup \dots \cup V_{r-1}$.  In the final step (where we apply Lemma \ref{ballbin} with $a_1=a_2=\Delta$ and $i=r$), either \ref{b2'} happens and we have a contradiction to $\nu_2(H)=\nu$, or \ref{b1'} happens and we have $\Delta(H) \geq \Delta\nu+1$, contradicting the assumption.  
\end{proof}

For the last result in this subsection we show that if there is a vertex of fairly large degree, then either we have an edge of multiplicity at least $\nu+1$ or there is a vertex of even larger degree.  

\begin{lemma}\label{lem:fatedge}
Let $r$ be an integer with $r\geq 3$ and let $s\in \{2,r-1,r\}$.  Let $H$ be an $r$-partite $r$-uniform multi-hypergraph with $n$ edges and set $\nu:=\nu_s(H)$.  If $\Delta(H)\geq  3^{\binom{r+1}{2}}\nu +1$, then either $H$ has an edge of multiplicity at least $\nu+1$ or 
\begin{enumerate}
\item if $s=2$, then $\Delta(H)\geq \frac{n-2\nu}{r-1}$.
\item if $s=r$, then $\Delta(H)\geq n-2\nu$.
\item if $s=r-1\geq 3$, then $\Delta(H)\geq \frac{(r-1)n}{r}-2(r-1)\nu$.
\end{enumerate}
\end{lemma}

\begin{proof}
Let $V_1, \dots, V_r$ be the parts of $H$.  For a set $U\subseteq V(H)$, let $d(U)$ denote the number of edges, counting multiplicity, which contain $U$ (i.e. $d(U)$ is the degree of $U$).  Note that since $H$ is $r$-partite, $d(U)>0$ implies that $U$ contains at most one vertex from each part $V_i$.  Let $U\subseteq V(H)$ be maximum such that 
$d(U)\geq 3^{\binom{r+2-|U|}{2}} \nu+1$ and note that $U\neq \emptyset$ by the degree condition.  Without loss of generality, suppose $U=\{u_1, \dots, u_\ell\}$ with $u_i\in V_i$ for all $i\in [\ell]$ and let $E$ be the set of edges containing $U$.  If $\ell=r$, we have an edge of multiplicity at least $3\nu+1\geq u+1$ and we are done; so suppose $1\leq \ell\leq r-1$.  

\noindent
\tbf{Case (i)} ($s=2$). Let $F=\{f\in E(H): f\cap U=\emptyset\}$.  If $|F|\leq \frac{(r-1-\ell)n+2\ell\nu}{r-1}$, then for some $i\in [\ell]$, $$d(u_i)\geq \frac{n-\frac{(r-1-\ell)n+2\ell\nu}{r-1}}{\ell} = \frac{n-2\nu}{r-1}$$ and we are done; so suppose $|F|\geq \frac{(r-1-\ell)n+2\ell\nu}{r-1}+1$.  

Applying Lemma \ref{ballbin_precise} at most $r-\ell$ times with $E, F$ and using the fact that $\nu_2(H)\leq \nu$, it must be the case that \ref{b1} holds within $r-\ell$ steps and we obtain a vertex which is incident with at least $\frac{3^{\binom{r+2-\ell}{2}-1}}{3^{r-\ell}}\nu+1=3^{\binom{r+2-(\ell+1)}{2}}\nu+1$ edges of $E$ contradicting the maximality of $U$.  

\noindent
\tbf{Case (ii) and (iii)} ($r-1\leq s\leq r$).

Since $|E|=d(U)\geq 3^{\binom{r+2-\ell}{2}}\geq 2(3^{\binom{r+2-\ell}{2}-1}\nu+1)$, we can choose disjoint subsets $E_1$ and $E_2$ of $E$, each with at least $3^{\binom{r+2-\ell}{2}-1} \nu+1$ edges.  Applying Lemma \ref{ballbin} at most $r-\ell$ times with $E_1,E_2$, we will either find a vertex which is contained in at least 
$\frac{3^{\binom{r+2-\ell}{2}-1}}{3^{r-\ell}}\nu+1=3^{\binom{r+2-(\ell+1)}{2}}\nu+1$
edges from both $E_1$ and $E_2$, which would violate the maximality of $U$, or else we will get sets $E_1'\subseteq E_1$ and $E_2'\subseteq E_2$ with 
\begin{equation}\label{E1'}
|E_1'|, |E_2'|\geq 3^{\binom{r+1-\ell}{2}}\nu+1\geq 3\nu+1 \text{ such that for all } e_1\in E_1', e_2\in E_2',~ e_1\cap e_2=U
\end{equation}
(where the last inequality holds since $\ell\leq r-1$).

\tbf{Case (ii)} ($s=r$).  We have the desired degree condition unless for all $i\in [\ell]$, the set $F_i$ of edges which avoids $u_i$ has order at least $2\nu+1$.  If $\ell\leq r-2$, then $E_1', E_2', F_1, \dots, F_\ell$ is a collection of $\ell+2\leq r$ sets each of order at least $\nu+1$ which are not cross intersecting, violating the bound on $\nu_r(H)$.  

So suppose $\ell=r-1$.  Applying Lemma \ref{ballbin_precise} with $E, F_1, \dots, F_\ell$ and using the fact that $\ell+1\leq r$ and $\nu_r(H)\leq \nu$, it must be the case that \ref{b1} holds and we obtain a vertex which is incident with at least $3\nu+1$; i.e. an edge of multiplicity at least $\nu+1$.

\tbf{Case (iii)} ($s=r-1\geq 3$).  We have the desired degree condition unless for all $i\in [\ell]$, the set $F_i$ of edges which avoids $u_i$ has order at least $\frac{n}{r}+2(r-1)\nu+1$.  If $\ell\leq r-3$, then $E_1', E_2', F_1, \dots, F_\ell$ is a family of $\ell+2\leq r-1$ sets of at least $\nu+1$ edges each which are not cross intersecting and thus $\nu_{r-1}(H)\geq \nu_{\ell+2}(H)\geq \nu+1$, contradicting the assumption.  If $\ell= r-2$, then applying Lemma \ref{ballbin_precise} at most twice with $E, F_1, \dots, F_\ell$ and using the fact that $\ell+1\leq r-1$ and $\nu_{r-1}(H)\leq \nu$, it must be the case that \ref{b1} holds within two steps and we obtain a vertex which is incident with at least $\frac{3^{\binom{4}{2}-1}}{3^{2}}\nu+1=3^3\nu+1$ edges of $E$ contradicting the maximality of $U$.  

So finally suppose $\ell=r-1$.  If there exists distinct $i,j\in [r-1]$ such that $|F_i\cap F_j|\geq 2\nu+1$, without loss of generality say $|F_{r-2}\cap F_{r-1}|\geq 2\nu+1$, then we apply Lemma \ref{ballbin_precise} with $E, F_1, \dots, F_{r-3}, F_{r-2}\cap F_{r-1}$ and since \ref{b2} can't happen, we have \ref{b1} which gives us an edge of multiplicity at least $\nu+1$.  So suppose $|F_i\cap F_j|\leq 2\nu$ for all distinct $i,j\in [r-1]$.  For all $i\in [r-1]$, let $F_i^*=F_i\setminus (\bigcup_{j\in [r-1]\setminus \{i\}}F_j)$ and note that by the previous sentence and the bound on $|F_i|$, we have $|F_i^*|\geq |F_i|-2(r-2)\nu \geq \frac{n}{r}+2\nu+1$.  Note that $F_1^*, \dots, F_{r-1}^*$ must be cross-intersecting and by the way the sets are defined, the cross-intersection must happen in $V_r$.  Now
 applying Lemma \ref{ballbin_precise} with $F_1^*, \dots, F_{r-1}^*$, we get a vertex in $V_r$ which is adjacent with at least $|F_i^*|-\nu\geq \frac{n}{r}$ edges from each of $F_1^*, \dots, F_{r-1}^*$ giving us a vertex of degree at least $\frac{r-1}{r}n\geq \frac{r-1}{r}n-2(r-1)\nu$ as desired.
\end{proof}

\subsection{Theorem \ref{dual:r2nu2} and Theorem \ref{dual:r3nu2}}\label{sec:nu2}

\begin{lemma}\label{fatedge3}
Let $r\geq 2$ and let $H$ be an $r$-partite multi-hypergraph with $n$ edges and set $\nu:=\nu_{2}(H)$.  If $H$ has an edge $e=\{u_1, \dots, u_r\}$ of multiplicity at least $\nu+1$, then 
\begin{enumerate}[label=\emph{(\roman*)}, ref=(\roman*)]
\item there are at least $n-\nu$ edges incident with $e$,
\item \label{ore'} for all $e'\subseteq e$ with $1\leq |e'|\leq r-1$, either the number of edges incident with every vertex in $e'$ and no vertex in $e\setminus e'$ is at most $\nu$, or the number of edges incident with every vertex in $e\setminus e'$ and no vertex in $e'$ is at most $\nu$, and 
\item \label{note'} either $\Delta(H)\geq \frac{n}{r-1}-2\nu$, or for all $e'\subseteq e$ with $1\leq |e'|=:t\leq r-1$, there are at least $\frac{(r-1-t)n}{r-1}+(2t-1)\nu+1$ edges incident with $e\setminus e'$ but not $e'$.
\end{enumerate}
\end{lemma}

\begin{proof}
Note that (i) and (ii) just follow from the condition on $\nu_2(H)$.  To see (iii), let $e'\subseteq e$ with $1\leq |e'|=:t\leq r-1$.  If the number of edges incident with $e'$ is at least $\frac{tn}{r-1}-2t\nu$, then some $u\in e'$ satisfies $d(u)\geq \frac{n}{r-1}-2\nu$ and we are done.  So suppose that $e'$ is incident with fewer than $\frac{tn}{r-1}-2t\nu$ edges, which means there are at least $$n-\nu-(\frac{tn}{r-1}-2t\nu)+1=\frac{(r-1-t)n}{r-1}+(2t-1)\nu+1$$ edges which are incident with $e\setminus e'$ but not $e'$.  
\end{proof}

Now we prove that if $H$ is a bipartite multigraph with $n$ edges and $\nu_2(H)<n/6$, then $\Delta(H)\geq n-2\nu_2(H)$.

\begin{proof}[Proof of Theorem \ref{dual:r2nu2}]
Let $V_1, V_2$ be the parts of $H$ and set $\nu:=\nu_2(H)<n/6$.  

\noindent
\textbf{Case 1} (There exists an edge $u_1u_2$ of multiplicity at least $\nu+1$).  By Lemma \ref{fatedge3}(i) and (ii), there are at least $n-\nu$ edges incident with $\{u_1,u_2\}$ and without loss of generality, there are at most $\nu$ edges which are incident with $u_2$ but not $u_1$.  Thus there are at least $n-2\nu$ edges incident with $u_1$; i.e. $\Delta(H)\geq n-2\nu$.

\noindent
\textbf{Case 2} (Every edge has multiplicity at most $\nu$).  Suppose first that there exists $u_1\in V_1, u_2\in V_2$ so that $d(u_1), d(u_2)\geq 2\nu+1$.  Since $u_1u_2$ has multiplicity at most $\nu$, there are at least $\nu+1$ edges incident with $u_1$ but not $u_2$ and at least $\nu+1$ edges incident with $u_2$ but not $u_1$, a violation of the fact that $\nu_2(H)\leq \nu$.  So suppose without loss of generality that 
\begin{equation}\label{maxV2}
d(u)\leq 2\nu \text{ for all } u\in V_2.
\end{equation}

Now let $V_2'\subseteq V_2$ be minimal such that $e(V_2', V_1)\geq 2\nu+1$.  By \eqref{maxV2} and minimality, we have $2\nu+1\leq e(V_2', V_1)\leq 4\nu$.  Since $6\nu<n$, we also have $e(V_2\setminus V_2', V_1) = n - e(V_2', V_1)\geq 2\nu+1$.  Furthermore, by pigeonhole and the fact that $6\nu<n$, we have either $e(V_2', V_1)\geq 3\nu+1$ or $e(V_2\setminus V_2', V_1)\geq 3\nu+1$.  So by applying Lemma \ref{ballbin_precise} (with $a_1=a_2=1$, $i=1$, $F_1=E(V_2', V_1)$, and $F_2=E(V_2\setminus V_2', V_1)$), we either have \ref{b2} (that is, there exists $F_1'\subseteq F_1$ with $|F_1'|\geq \nu+1$ and $F_2'\subseteq F_2$ with $|F_2'|\geq \nu+1$ such that $F_1'$ and $F_2'$ are not cross intersecting in $V_1$) which violates the fact that $\nu_2(H)\leq \nu$, or \ref{b1} which implies that there exists a vertex in $V_1$ which is incident with at least $|F_1|+|F_2|-2\nu=n-2\nu$ edges; i.e. $\Delta(H)\geq n-2\nu$.  
\end{proof}

\begin{proposition}\label{fatedgetodegree3}
Let $H$ be an $3$-partite $3$-uniform multi-hypergraph with $n$ edges and set $\nu:=\nu_{2}(H)$.  If $H$ has an edge  of multiplicity at least $\nu+1$, then $\Delta(H)\geq \frac{n}{2}-2\nu$.
\end{proposition}

\begin{proof}
Let $e=\{u_1,u_2, u_3\}$ be an edge of multiplicity at least $\nu+1$.  
For all distinct $i,j,k\in [3]$, let $E_i$ be the set of edges incident with $u_i$ and let $E_i'=E_i\setminus (E_j\cup E_k)$.
By Lemma \ref{fatedge3}.\ref{note'} we have $|(E_1\cup E_2)\setminus E_3|\geq \frac{n}{2}+\nu+1$ and for all $i\in [3]$, $|E_i'|\geq 3\nu+1$.  So by Lemma \ref{fatedge3}.\ref{ore'}, $|(E_1\cap E_2)\setminus E_3|\leq \nu$.  Thus $|E_1'|+|E_2'|\geq \frac{n}{2}+1$.  Now applying Lemma \ref{ballbin} with $E_1'$ and $E_2'$, we can't have \ref{b2'}, thus \ref{b1'} holds and we have a vertex in $V_3$ which is incident with more than $\frac{n}{2}-2\nu$ edges.
\end{proof}

Now we prove that if $H$ is a $3$-partite $3$-uniform multi-hypergraph with $n$ edges and $\nu_{2}(H)\leq \frac{n}{3^9}=\frac{n}{19683}$, then $\Delta(H)\geq \frac{n}{2}-2\nu_2(H)$.

\begin{proof}[Proof of Theorem \ref{dual:r3nu2}]
Set $\nu:=\nu_2(H)$. By Lemma \ref{decentdegree} (with $\Delta=729=3^6$), we have $\Delta(H)\geq 729 \nu +1=3^{\binom{3+1}{2}}\nu+1$.  Now by Lemma \ref{lem:fatedge}, we are done or we have an edge of multiplicity at least $\nu+1$ in which case we are done by Proposition \ref{fatedgetodegree3}.
\end{proof}

In this subsection we solved Problem \ref{prob:gen} in the case $k=2$ and $2\leq r\leq 3$.  Because of Lemma \ref{decentdegree} and Lemma \ref{lem:fatedge}, in order to solve Problem \ref{prob:gen} in the case $k=2$ and $r\geq 4$ it suffices to prove the following generalization of Proposition \ref{fatedgetodegree3}.  

\begin{conjecture}\label{con:k=2}
Let $r\geq 4$ and let $H$ be an $r$-partite $r$-uniform multi-hypergraph with $n$ edges and set $\nu:=\nu_{2}(H)$.  There exists $d_r>0$ such that if $H$ has an edge of multiplicity at least $\nu+1$, then $\Delta(H)\geq \frac{n}{r-1}-d_r\nu$.
\end{conjecture}

Finally we give a proof of Proposition \ref{propweak} (which is essentially just a much weaker version of Conjecture \ref{con:k=2}).

\begin{proposition}
Let $r\geq 4$ and let $H$ be an $r$-partite $r$-uniform multi-hypergraph with $n$ edges.  If $\nu_{2}(H)\leq \frac{n}{3^{\binom{r+1}{2}+r}}$, then $\Delta(H)\geq \frac{n-\nu_2(H)}{r}$.
\end{proposition}

\begin{proof}
Set $\nu:=\nu_{2}(H)$.
By Lemma \ref{decentdegree} we have $\Delta(H)\geq 3^{\binom{r+1}{2}}\nu+1$.  Now by Lemma \ref{lem:fatedge}, we are done or we have an edge $e$ of multiplicity at least $\nu+1$.  Thus by Lemma \ref{fatedge3}.(i), we have $n-\nu$ edges incident with $e$ so, by averaging, one of these vertices has degree at least $\frac{n-\nu}{r}$.
\end{proof}

\subsection{Theorem \ref{dual:rrnur}}\label{sec:nur}

\begin{proposition}\label{fatedgetodegree}
Let $r\geq 3$ and let $H$ be an $r$-partite $r$-uniform hypergraph with $n$ edges and set $\nu:=\nu_r(H)$.  If $H$ has an edge of multiplicity at least $\nu+1$, then $\Delta(H)\geq n-(r-1)\nu$.
\end{proposition}

\begin{proof}
Assume there exists an edge $e=\{u_1, \dots, u_r\}$ of multiplicity at least $\nu+1$.  For all $i\in [r]$, let $F_i$ be the set of edges which avoid $u_i$.  If $|F_i|\leq (r-1)\nu$ for some $i\in [r]$, then $d(u_i)\geq n-(r-1)\nu$ and we are done; so suppose $|F_i|\geq (r-1)\nu+1$ for all $i\in [r]$.

\begin{claim}\label{clm:2way}
For all distinct $i,j\in [r]$, $|F_i\cap F_j|\leq \nu$.
\end{claim}

\begin{proofclaim}
Suppose for contradiction that $|F_i\cap F_j|\geq \nu+1$ for some distinct $i,j\in [r]$ and without loss of generality suppose $\{i,j\}=[2]$.  Now $e, F_1\cap F_2, F_3,  \dots, F_r$ is a collection of $r$ sets violating $\nu_{r}(H)\leq \nu$.  
\end{proofclaim}

Now for all $i\in [r-1]$, let $F_i'=F_i\setminus \bigcup_{j\in [r]\setminus\{i,i+1\}}F_j$ and let $F_r'=F_r\setminus \bigcup_{j\in [r]\setminus\{r,1\}}F_j$.  Note that by Claim \ref{clm:2way} we have that for all $i\in [r]$, $|F_i'|\geq (r-1)\nu+1-(r-2)\nu=\nu+1$.  Furthermore, by construction, we have $F_i'\cap F_j'=\emptyset$ for all distinct $i,j\in [r]$.  So we have $r$ disjoint sets $F_1', \dots, F_r'$ each of order at least $\nu+1$ which are not cross intersecting, violating the assumption.  Indeed, let $e_i\in F_i'$ for all $i\in [r]$ and suppose for contradiction that $\bigcap_{i\in [r]}e_i\neq \emptyset$.  Let $u\in \bigcap_{i\in [r]}e_i$ and suppose without loss of generality that $u\in V_1$.  We cannot have $u=u_1$ since $e_1\in F_1'\subseteq F_1$ misses the vertex $u_1$, but also we cannot have $u\neq u_1$ since $e_2\in F_2'$ and $F_2'\cap F_1=\emptyset$ and thus $e_2$ touches $u_1$.
\end{proof}

Now we prove that if $r\geq 3$ and $H$ is an $r$-partite $r$-uniform hypergraph with $n$ edges and $\nu_r(H)\leq \frac{n}{3^{\binom{r+1}{2}+r}}$, then $\Delta(H)\geq n-(r-1)\nu_r(H)$.

\begin{proof}[Proof of Theorem \ref{dual:rrnur}]
Set $\nu:=\nu_r(H)$.
By applying Lemma \ref{decentdegree} with $\Delta=3^{\binom{r+1}{2}}$, we have $\Delta(H)\geq 3^{\binom{r+1}{2}} \nu +1$.  Now by Lemma \ref{lem:fatedge}, we are done or we have an edge of multiplicity at least $\nu+1$ in which case we are done by Proposition \ref{fatedgetodegree}.
\end{proof}

\subsection{Theorem \ref{dual:rrnur-1}}\label{sec:nur-1}

\begin{proposition}\label{fatedgetodegreer}
Let $r\geq 4$ and let $H$ be an $r$-partite multi-hypergraph with $n$ edges and set $\nu:=\nu_{r-1}(H)$.  If $H$ has an edge of multiplicity at least $\nu+1$, then $\Delta(H)\geq \frac{r-1}{r}n-\binom{r}{2}\nu$.
\end{proposition}

\begin{proof}
Assume there exists an edge $e=\{u_1, \dots, u_r\}$ of multiplicity at least $\nu+1$.  For all $i\in [r]$, let $F_i$ be the set of edges which avoid $u_i$.  
If $|F_i|\leq \frac{n}{r}+\binom{r}{2}\nu$ for some $i\in [r]$, then $\Delta(H)\geq \frac{r-1}{r}n-\binom{r}{2}\nu$ and we are done; so suppose $|F_i|> \frac{n}{r}+\binom{r}{2}\nu \ge (r-1)\nu +1$ for all $i\in [r]$.
Let $F=F_1\cup \dots \cup F_r$.

\begin{claim}\label{clm:3way}
For all distinct $h,i,j\in [r]$, $|F_h\cap F_i\cap F_j|\leq \nu$.
\end{claim}

\begin{proofclaim}
Suppose for contradiction that $|F_h\cap F_i\cap F_j|\geq \nu+1$ for some distinct $h,i,j\in [r]$ and without loss of generality suppose $\{h,i,j\}=[3]$.  Now $e, F_1\cap F_2\cap F_3, F_4, \dots, F_r$ is a collection of $r-1$ sets violating $\nu_{r-1}(H)\leq \nu$.  
\end{proofclaim}

\begin{claim}\label{clm:2way2}
For all distinct $h,i,j,k\in [r]$, $|F_h\cap F_i|\leq \nu$ or $|F_j\cap F_k|\leq \nu$.
\end{claim}

\begin{proofclaim}
Suppose for contradiction that $|F_h\cap F_i|\geq \nu+1$ and $|F_j\cap F_k|\geq \nu+1$ for some distinct $h,i,j,k\in [r]$ and without loss of generality suppose $\{h,i,j,k\}=[4]$.  Now $e, F_1\cap F_2, F_3\cap F_4, F_5, \dots, F_r$ is a collection of $r-1$ sets violating $\nu_{r-1}(H)\leq \nu$.  
\end{proofclaim}

For all $S\subseteq [r]$, let $\left(\cap_{i\in S}F_i\right)^*=\left(\cap_{i\in S}F_i\right)\setminus \left(\cup_{j\in [r]\setminus S}F_j\right)$.  In other words $\left(\cap_{i\in S}F_i\right)^*$ is the collection of elements which are in all of the sets $F_i$, $i\in S$, but none of the other sets $F_j$, $j\in [r]\setminus S$.  

\begin{claim}\label{clm:1way2way}
For all distinct $h,i,j\in [r]$, $|(F_h\cap F_i)^*|\leq \nu$ or $|F_j^*|\leq \nu$.
\end{claim}

\begin{proofclaim}
Suppose for contradiction that $|(F_h\cap F_i)^*|\geq \nu+1$ and $|F_j^*|\geq \nu+1$ for some distinct $h,i,j\in [r]$ and without loss of generality suppose $\{h,i,j\}=[3]$.  Now the sets $(F_1\cap F_2)^*, F_3^*, F_4, \dots, F_r$ is a collection of $r-1$ sets violating $\nu_{r-1}(H)\leq \nu$.
\end{proofclaim}

Since $|F_i|> \frac{n}{r}+\binom{r}{2}\nu $ for all $i\in [r]$, 
inclusion-exclusion implies that $|F_i\cap F_j|\geq (r-1)\nu+1$ for some distinct $i,j\in [r]$; without loss of generality, say $i=r-1$ and $j=r$.  Furthermore, by Claim \ref{clm:3way} we must have that $|(F_{r-1}\cap F_{r})^*|\geq \nu+1$.  Thus by Claim \ref{clm:2way2}, we have that for all distinct $i,j\in [r-2]$, $|F_i\cap F_j|\leq \nu$, and by Claim \ref{clm:1way2way} we have that for all $i\in [r-2]$, $|F_i^*|\leq \nu$.

So for all $i\in [r-2]$, we have $|F_i\setminus (F\setminus F_i)|\leq \nu$, $|F_i\cap F_{r-1}\cap F_r|\leq \nu$, and for all $j\in [r-2]\setminus\{i\}$, $|F_i\cap F_j|\leq \nu$, thus 
$$|F_i\cap F_{r-1}|+|F_i\cap F_r|\geq |F_i|-(r-1)\nu\geq \frac{n}{r}+\binom{r}{2}\nu-(r-1)\nu\geq \frac{n}{r}+2\nu.$$  
Let $i\in [r-2]$.  Without loss of generality, suppose $|F_i\cap F_r|\geq \frac{1}{2}(|F_i\cap F_{r-1}|+|F_i\cap F_r|)\geq \frac{n}{2r}+\nu>\nu$.  Thus by Claim \ref{clm:2way2} we have that for all $j\in [r-2]\setminus\{i\}$, $|F_j\cap F_{r-1}|\leq \nu$ which in turn implies that for all $i\in [r-2]$, $|(F_i\cap F_r)^*|\geq |F_i|-r\nu>\frac{n}{r}+\nu$.
 By Claim \ref{clm:1way2way} this implies that $|F_{r-1}^*|\leq \nu$.  Thus $|(F_{r-1}\cap F_{r})\setminus (F_1\cup\dots \cup F_{r-2})|\geq |F_{r-1}|-\nu-(r-2)\nu\geq \frac{n}{r}+\nu$.  Now we have a collection of $r-1$ sets, $(F_1\cap F_r)^*, \dots, (F_{r-2}\cap F_r)^*, (F_{r-1}\cap F_{r})^*$ all with more than $\frac{n}{r}+\nu$ elements.  Now applying Lemma \ref{ballbin} (with $a_1=\dots=a_{r-1}=1$) to the collection of $r-1$ sets, we cannot have \ref{b2'} by the bound on $\nu_{r-1}(H)$, so we must have \ref{b1'} which gives us a vertex in $V_r\setminus \{u_r\}$ with degree at least $(r-1)((\frac{n}{r}+\nu)-\nu)=\frac{r-1}{r}n$. 
\end{proof}

Now we prove that if $r\geq 3$ and $H$ is an $r$-partite $r$-uniform hypergraph with $n$ edges and $\nu_{r-1}(H)\leq \frac{n}{3^{\binom{r+1}{2}+r}}$, then $\Delta(H)\geq \frac{(r-1)n}{r}-\binom{r}{2}\nu_{r-1}(H)$.

\begin{proof}[Proof of Theorem \ref{dual:rrnur-1}]
Set $\nu:=\nu_{r-1}(H)$.  
By Lemma \ref{decentdegree} (with $\Delta=3^{\binom{r+1}{2}}$), we have $\Delta(H)\geq 3^{\binom{r+1}{2}} \nu +1$.  Now by Lemma \ref{lem:fatedge} we are done, or we have an edge of multiplicity at least $\nu+1$ in which case we are done by Proposition \ref{fatedgetodegreer}.
\end{proof}

\section{Conclusion}

We were able to solve Problem \ref{prob:gen} in all cases corresponding to Theorem \ref{thm:Gy} except when $k=2$ and $r\geq 4$.  However because of Lemma \ref{decentdegree} and Lemma \ref{lem:fatedge}, in order to solve the case $k=2$ and $r\geq 4$ it suffices to prove Conjecture \ref{con:k=2}.  It would be very interesting to prove Conjecture \ref{con:k=2} even in the case $k=4$.  

Another possible direction for further study involves replacing large monochromatic components with long monochromatic paths.  Letzer showed that in every 2-coloring of $G(n,p)$ with $p=\frac{\omega(1)}{n}$, there is w.h.p., a monochromatic cycle (path) of order at least $(2/3-o(1))n$.  Bennett, DeBiasio, Dudek, and English \cite{BDDE} generalized this result showing that if $p=\frac{\omega(1)}{n^{k-1}}$, then a.a.s. there is a monochromatic loose-cycle (loose-path) of order at least $(\frac{2k-2}{2k-1}-o(1))n$ in every 2-coloring of $H^k(n,p)$.  Both of those results use sparse regularity and implicitly only use the fact $\alpha_k(G)=o(n)$, so we can retroactively rephrase their result as follows.

\begin{theorem}[Bennett, DeBiasio, Dudek, and English \cite{BDDE}]
If $G$ is a $k$-uniform hypergraph on $n$ vertices with $\alpha_k(H)=o(n)$, then in every $2$-coloring of the edges of $G$, there exists a monochromatic loose-cycle (loose-path) of order at least $(\frac{2k-2}{2k-1}-o(1))n$.
\end{theorem}

The idea is that it would be nice to extend the above theorem to hold when $\alpha_k(G)$ can be considerably larger (especially in the case $k=2$).

There are two results in the literature which implicitly broach this subject.  Balogh, Bar\'at, Gerbner, Gy\'arf\'as, S\'ark\"ozy \cite{BBGGS} proved that in every 2-coloring of the edges of a graph $G$ on $n$ vertices there exist two vertex disjoint monochromatic paths covering at least $n-1000(50\alpha_2(G))^{\alpha_2(G)}$ vertices.  Letzter \cite{L} implicity proved that in every 2-coloring of every graph $G$ on $n$ vertices there is a monochromatic path of order at least $\frac{n}{2}-2\alpha_2(G)$.  

So a particular case of the general problem we are interested in is the following. 

\begin{problem}
Given $n$ sufficiently large, determine the largest value of $\alpha$ such that if $G$ is a graph on $n$ vertices with $\alpha_2(G)\leq \alpha$, then in every 2-coloring of $G$ there is a monochromatic path of order greater than $n/2$.
\end{problem}

Finally, we mention that the best upper bounds on the size-Ramsey number of a path come from random $d$-regular graphs $G(n,d)$ (see \cite{DP1}).  An upper bound on $\mc_2(G(n,d))$ would give an upper bound on the longest monochromatic path.  However, determining an upper bound on the largest monochromatic component in an arbitrary 2-coloring of $G(n,d)$ for small $d$ falls outside the purview of this paper (partly since $\alpha_2$ can be large in this case).  So we raise the following problem.  

\begin{problem}
Determine bounds on $\mc_2(G(n,d))$ for $d\geq 5$.  More generally, determine bounds on $\mc_r(G(n,d))$ for $d\geq 2r+1$.
\end{problem}

Note that a result of Anastos and Bal \cite{AB} implies that $\mc_r(G(n,d))=o(n)$ when $d\leq 2r$.


\begin{thebibliography}{99}

\bibitem{AB} M. Anastos, D. Bal.  A Ramsey property of random regular and $k$‐out graphs.  \emph{Journal of Graph Theory} \tbf{93}, no. 3 (2020): 363-371.

\bibitem{BD1} D. Bal, L. DeBiasio.  Partitioning random graphs into monochromatic components. \emph{Electronic Journal of Combinatorics} \tbf{24}, no. 1 (2017): P1.18

\bibitem{BD2} D. Bal, L. DeBiasio.  Large monochromatic components in hypergraphs with large minimum codegree. arXiv preprint arXiv:2301.05806 (2023).

\bibitem{BBGGS} J. Balogh, J. Bar\'at, D. Gerbner, A. Gy\'arf\'as, G. S\'ark\"ozy.  Partitioning 2-edge-colored graphs by monochromatic paths and cycles. \emph{Combinatorica} \tbf{34}, no. 5 (2014): 507-526.

\bibitem{BDDE} P. Bennett, L. DeBiasio, A. Dudek, S. English. Large monochromatic components and long monochromatic cycles in random hypergraphs. \emph{European Journal of Combinatorics} \tbf{76} (2019): 123-137.

\bibitem{DT} L. DeBiasio, M. Tait. Large monochromatic components in 3‐edge‐colored Steiner triple systems. \emph{Journal of Combinatorial Designs} \tbf{28}, no. 6 (2020): 428-444.

\bibitem{DP1} A. Dudek, P. Pra{\l}at.  An Alternative Proof of the Linearity of the Size-Ramsey Number of Paths. \emph{Combinatorics, Probability \& Computing} \tbf{24}, no. 3 (2015): 551-555.

\bibitem{DP2} A. Dudek, P. Pra{\l}at.  On some multicolor Ramsey properties of random graphs. \emph{SIAM Journal on Discrete Mathematics} \tbf{31}, no. 3 (2017): 2079-2092.

\bibitem{F} Z. F\"uredi.  Maximum degree and fractional matchings in uniform hypergraphs. \emph{Combinatorica} \tbf{1}, no. 2 (1981): 155-162.

\bibitem{FL} Z. F\"uredi, R. Luo.  Large Monochromatic Components in Almost Complete Graphs and Bipartite Graphs.  \emph{The Electronic Journal of Combinatorics} (2021): P2.42.


\bibitem{GSc} H. Guggiari, A. Scott. Monochromatic Components in Edge-Coloured Graphs with Large Minimum Degree. \emph{The Electronic Journal of Combinatorics} (2021): P1.10.

\bibitem{Gy} A. Gy\'arf\'as. Partition coverings and blocking sets in hypergraphs. \emph{Communications of the Computer and Automation Institute of the Hungarian Academy of Sciences} \tbf{71} (1977): 62.

\bibitem{Gy2} A. Gy\'arf\'as. Large cross‐free sets in Steiner triple systems. \emph{Journal of Combinatorial Designs} \tbf{23}, no. 8 (2015): 321-327.


\bibitem{GS2} A. Gy\'arf\'as, G.N. S\'ark\"ozy. Large monochromatic components in edge colored graphs with a minimum degree condition. \emph{The Electronic Journal of Combinatorics} \tbf{24}, no. 3 (2017): 3-54.

\bibitem{KS} M. Krivelevich, B. Sudakov. The chromatic numbers of random hypergraphs. \emph{Random Structures Algorithms} \tbf{12}, no. 4 (1998): 381–403.

\bibitem{L} S. Letzter. Path Ramsey number for random graphs. \emph{Combinatorics, Probability and Computing} \tbf{25}, no. 4 (2016): 612-622.

\bibitem{R} Z. Rahimi.  Large monochromatic components in 3-colored non-complete graphs.  \emph{Journal of Combinatorial Theory, Series A} \tbf{175} (2020): 105256.

\end{thebibliography}
\end{document}